\documentclass[a4paper,11pt]{amsart}
 \usepackage{amssymb,amsfonts,amsthm,amscd,stmaryrd,amsmath}
\usepackage[utf8]{inputenc}
\usepackage[foot]{amsaddr}
\usepackage[dvipsnames]{xcolor} 
\usepackage{graphicx}
\usepackage{tikz}
\usetikzlibrary{decorations.pathreplacing}
\usepackage{subcaption}
\usepackage{color}

\usepackage{pgf,tikz}
 \usetikzlibrary{arrows}
\usepackage[colorlinks=true,urlcolor=blue,
citecolor=red,linkcolor=blue,linktocpage,pdfpagelabels,
bookmarksnumbered,bookmarksopen]{hyperref}
 
\usepackage{ae}
\usepackage[utf8]{inputenc}
\numberwithin{figure}{section}

\usepackage{mathrsfs,a4wide}
\usepackage{esint}
\usepackage{color}

\usepackage{import}

\usepackage[leqno]{amsmath}

\newtheorem{thm}{Theorem}[section]
\newtheorem{lem}[thm]{Lemma}
\newtheorem{prop}[thm]{Proposition}
\newtheorem{cor}[thm]{Corollary}
\theoremstyle{definition}

\newtheorem{exa}[thm]{Example}

\newtheorem{oss}[thm]{Remark}

\numberwithin{equation}{section}
\font\script=rsfs10 at 12pt

    \def\res{\mathop{\hbox{\vrule height 7pt width .5pt depth 0pt \vrule height .5pt width 6pt depth 0pt}}\nolimits}

\newcommand{\beq}{\begin{equation}}
\newcommand{\eeq}{\end{equation}}
\newcommand{\Div}{\operatorname{div}}

\def \He {\mbox{\script{H}}\;}
\def\R{\mathbb R}

\def\H{\mathcal H}
\def\S{\mathbb S}
\def\e{\varepsilon}
\def\pa{\partial}

\def\restrict#1{\raise-.5ex\hbox{\ensuremath|}_{#1}}
\title{Weighted}

\author{Nicola Fusco \and Domenico Angelo La Manna}
\title{A remark on a conjecture on the symmetric Gaussian Problem}
\begin{document}
\maketitle
\section{Introduction}

Let $\gamma(E)$ and $P_\gamma(E)$
be the Gaussian measure and the Gaussian perimeter of a set $E\subset \R^n$, see the definition in the next section.
A classical result states that  that among all sets of given Gaussian measure the half space is the only one minimizing the Gaussian perimeter \cite{Bo, ST}.  
This inequality has been established also in a quantitative form with different approaches (see \cite{bbj},\cite{cfmp}),
 both of independent interest. Due to its large number of applications also the nonlocal form of the Gaussian isoperimetric inequality has been studied. In particular, it is known that the half space is still a minimizer of the fractional Gaussian perimeter if we define it by means of the Stinga-Torrea extension (see \cite{NPS,CCLP}), while this is no longer the case if one defines 
the fractional perimeter by means of a singular integral with a Gaussian weight (see \cite{DeL}).
On the contrary, if one restricts to the class of symmetric sets, 
the characterization of perimeter minimizers under a volume constraint is still an open problem.
 In \cite{Barthe} it was conjectured that these minimizers would be either the ball or its complement. This conjecture has been disproved in \cite{l}, see also \cite{Heil2}, and in \cite{BaJu} it has been proven that for small values of the volume the minimizer among symmetric sets is given by a strip.
Later on, in \cite{Heil} it was shown that jf such minimizer is convex and satisfies some addictional condition, then it is a round cylinder. Finally in a very recent paper Heilman, adapting previous ideas introduced by Colding and Minicozzi in the study of the mean curvature flow, proved  
that if the minimizer of the Symmetric Gaussian Problem (SGP) is of the form
$\Omega\times \R\subset \R^{n+1}$, then either $\Omega$ or $\Omega^c$ is convex
\cite{Heil3}.

In this paper we study a problem related to the above SGP, raised in \cite{Galyna}, where it was asked whether the ball centered at the origin maximizes the energy functional
\[
\He(E)=\int_{\pa E} H_{\pa E}e^{-\frac{|x|^2}{2}} \, d\H^{n-1}
\]
among all convex, symmetric sets with fixed Gaussian volume. Here and in the following, $H_{\pa E}$ denotes the mean curvature of $E$. 
More precisely, we prove that this is true in two dimensions, even if one replaces $e^{-\frac{|x|^2}{2}}$ with more general weights on the volume and on the perimeter. Furthermore, we show that this maximality property of the disk holds in a stronger form,
in the sense that if $\gamma(E)=\gamma(B_r)$ the gap $\He(B_r)-\He(E)$ can be estimated in a quantitative way both from above and from below, see Theorem \ref{teo:2d}.
The estimate from above follows by a simple integration by parts, while the one from below follows by a  calibration argument.
\\
The situation is completely different in higher dimension. A simple example given in Section \ref{sec:highd} shows that already in dimension three one can find a symmetric cylinder $E$
such that $\gamma(E)=\gamma(B_r)$, but 
 $\He(E)>\He(B_r)$ for sufficiently small values of $r$. Even worse, in Theorem \ref{teo:localmax} we prove that if $E$ is a symmetric set sufficiently close in $W^{2,\infty}$ to the ball $B_r$ with the same Gaussian volume, then $\He(B_r)<\He(E)$ provided that $r^2< \frac{(n-2)(n-1)}{2n}$. On the other hand, if $r^2>n-2$  we are able to show that the ball is a local maximizer of $\He$ with respect to all competitors $E$ close in $W^{2,\infty}$, with the same Gaussian volume and not necessarily symmetric.  
These two local minimality and maximality results are obtained with a more or less standard second order Taylor expansion of the functional $\He$ around the ball. However, a completely different calibration argument allows us to show 
the maximality property of the ball among all convex sets satisfying a uniform bound on the curvature, see Theorem \ref{teo:localmax1}.
This latter result suggests that for sufficiently large Gaussian volume, the ball shound be indeed a global maximizer, but at the moment this seems to be a difficult open problem.  
\section{Preliminary Definitions}
For a measurable set $E\subset \R^n$, we denote by $\gamma(E)$ its Gaussian meausure 
\beq \label{eq:defmeas}
\gamma(E)= \frac{1}{(2\pi)^\frac n2}\int_E e^{-\frac{|x|^2}{2}}\, dx,
\eeq
normalized so that $\gamma(\R^n)=1$. 
We say that
$E$ has finite Gaussian perimeter if 
\[
P_\gamma(E)=\frac{1}{(2\pi)^{\frac{n-1}{2}}}\sup_{\|X\|_{L^\infty(\R^n)}\leq 1}\left\{\int_{\Omega}\Div(e^{-\frac{|x|^2}{2}} X(x))\, dx, \,\, X\in C^1_c(\R^n)\right\}<\infty.
\]
Note that if $E$ is a smooth set, then
\[
P_\gamma(E)= \frac{1}{(2\pi)^\frac{n-1}{2}}\int_{\pa E}e^{-\frac{|x|^2}{2}}\, d\H^{n-1},
\]
where $\H^{n-1}$ denotes the $(n-1)-$dimensional Hausdorff measure.

Let $E$ be an open set  of class $C^2$ and $X: M\to \R^n$ a $C^1$ vector field. 
For any $x \in M$, denoting by  $\tau_1,\dots,\tau_{n-1}$
an orthonormal base for the tangent space $T_xM$ with $x\in M$, the tangential divergence of $X$ is given by
\[
\Div_{\tau} X= \sum_{i=1}^{n-1} \langle \nabla_{\tau_i} X, \tau_i\rangle,
\]
where $\nabla_{\tau_i} X$ is the derivative of $X$ in the direction $\tau_i$.
Note that if we still denote by $X$ a $C^1$ extension of the vector field in a tubular neighborhood of $\pa E$, then
\[
\Div_\tau X= \Div X -\langle D X \nu_{\pa E},\nu_{\pa E}\rangle
\]
where $\nu_{\pa E}$ is the exterior normal to $E$.
We recall also that the mean curvature of $\pa E$ (actually the sum of the principal curvatures), is given by
\beq \label{eq:defnH}
H_{\pa E}= \Div_{\tau} \nu_{\pa E}.
\eeq
If we extend $\nu_{\pa E}$ in a tubular neighborhood of $\pa E$ so that the resulting vector field is still of class $C^1$, then $H_{\pa E}=\Div \nu_{\pa E}$ on $\pa E$.

Observe that with this definition it turns out that if $E$ is locally the subgraph of a $C^2(\R^{n-1})$ function $u$ then 
\beq \label{eq:graph}
H_{\pa E}=- \Div\left(\frac{D u}{\sqrt{1+|D u|^2}}\right).
\eeq
We recall that if $E$ is a bounded open set of class $C^2$ and $X\in C^1(\pa E,\R^n)$, the divergence theorem for manifolds states that
\[
\int_{\pa E} \Div_{\tau } X d\H^{n-1}= \int_{\pa E} H_{\pa E}\langle X,\nu_{\pa E} \rangle \, d\H^{n-1}.
\]
In particular, if $X$ is a tangent vector field it holds
\[
\int_{\pa E} \Div_{\tau } X d\H^{n-1}=0.
\]
Note that if $E$ is an open set of class $C^{1,1}$, hence it is locally the subgraph of a $C^{1,1}
$ function $u$, the mean curvature of $\pa E$ can be defined using \eqref{eq:graph}. With this definition the above divergence theorem still holds.
Finally, the Laplace-Beltrami operator on $\pa E$ is defined for any $h\in C^2(\pa E)$ as 
\[
\Delta_{\pa E} h= \Div_{\tau} \nabla h
\]
where $\nabla h$ denotes the tangential gradient of $h$.
\section{Two dimensional case: a two side estimate of the integral of the curvature in weighted spaces}
\label{sec:2d}
In this section we provide an estimate for the weighted integral of the curvature under suitable assumptions on the weight.
Let $f:[0,\infty)\to (0,\infty) $ be a $C^1$ not increasing function and $w:(0,+\infty)\to [0,\infty)$ be defined as
\beq \label{eq:defw}
w(r)= - \frac{ f'(r)}{r}.
\eeq 
We define the weighted area $|E|_w$ of a set $E$ as
\[
|E|_w= \int_E w(|x|)\, dx.
\]
Note that if $f=e^{-\frac{r^2}{2}}$ then $w= e^{-\frac{r^2}{2}}$. Hence the results given in this section apply to the particular case of the Gaussian weight.

We start by proving an isoperimetric type inequality concerning a weighted integral of the curvature. To this aim, here and in the following we denote by $B_r$ the ball centered at the origin with radius $r$.
\begin{prop} \label{prop:curv}
Let $r>0$, $f:[0,\infty)\to (0,\infty) $ a $C^1$ not increasing function and
 $w$ be defined as in \eqref{eq:defw}.
For any convex set $E\subset \R^2$ of class $C^{1,1}$ with $|E|_w=|B_r|_w$  containing the origin it holds
\beq \label{eq:curvature2d}
\int_{\pa E}H_{\pa E}f(|x|)\,d\H^1\leq
\int_{\pa B_r} H_{\pa B_r}f(|x|)\, d\H^1.
\eeq
If $w$ is not increasing  \eqref{eq:curvature2d} holds for any convex set $E$ of class $C^{1,1}$  with $|E|_w=|B_r|_w$.  
\end{prop}
\begin{proof}
For a convex set $E$ of class $C^{1,1}$ containing the origin 
we denote by $\rho: \R\to (0,\infty)$ a $C^{1,1}$ periodic function such that $\pa E= \{\rho(\theta)( \sin \theta, \cos \theta):\, \theta \in [0,2\pi]\}$.
Note that for almost every $\theta \in [0,2\pi]$ the curvature at $\rho(\theta)(\sin \theta, \cos \theta)$ is given by 
\[
H_{\pa E}= \frac{\rho^2+2\rho'^2-\rho\rho''}{(\rho^2+\rho'^2)^\frac32}.
\]
Thus we compute 
\[
|E|_w= \int_0^{2\pi} \int_{0}^{\rho} t w(t)\, dt
\]
and
\[
\int_{\pa E} H_{\pa E} f(|x|)\,d\H^{1}
=\int_{0}^{2\pi} \frac{\rho^2+2\rho'^2-\rho\rho''}{\rho^2+\rho'^2}f(\rho)\,d\theta.
\]
Since $|E|_w=|B_r|_w$ we have, recalling \eqref{eq:defw}
\[
2\pi f(0)-\int_0^{2\pi} f(\rho)\, d\theta =\int_0^{2\pi} \int_{0}^{\rho} t w(t)\, dt= 2\pi \int_{0}^{r} t w(t)\, dt= 2\pi(f(0)- f(r))
\]
which gives
\[
\int_0^{2\pi} f(\rho)\, d\theta= 2\pi f(r).
\]
Hence, integrating by parts,
\beq \label{eq:curvature1}
\begin{split}
    \int_{\pa E}H_{\pa E} f(|x|) d\H^1&=
    \int_0^{2\pi} \frac{\rho'^2-\rho\rho''}{\rho^2+\rho'^2}  f(\rho)
+\int_ 0^{2\pi} f(\rho) \, d\theta 
\\
&=-\int_0^{2\pi} \frac{d}{d\theta}\arctan\left(\frac{\rho'}{\rho}\right) f(\rho) d\theta + 
\\ 
&=\int_{0}^{2\pi} (\rho \rho') \arctan\left(\frac{\rho'}{\rho}\right)\ f'(\rho) d\theta +2\pi f(r)\\
&\leq2\pi f(r)=\int_{\pa B_r}H_{\pa B_r}f(r) d\H^1 ,
\end{split}
\eeq
where in the last inequality we used that $t\arctan t\geq 0$ for all $t\in \R$ and $f'(\rho) \leq 0$ .

When $0\in \pa E$, given $\e>0$ small we may translate $E$ to get a set $E_\e= E+x_\e $ with $|x_\e|<\e$ and such that $0\in \operatorname{int} E_\e$. 
Then the validity of \eqref{eq:curvature2d} for $E$ follows by applying the same inequality to $E_\e$ and then letting $\e \to 0$. \\
In the case that the set $E$ does not contain the origin
let $x_0$ be the nearest point of $\overline{E}$ to the origin.
Hence, since $E$ is convex we have that 
$\langle x,x_0\rangle \geq |x_0|^2$ for all $x\in \overline{E}$, which in turn implies that $|x-x_0|< |x| $ for all $x\in \overline{E} $.
Thus $ |E-x_0|_w>|E|_w$. Let $s>0$ such that $|B_s|_w=|E-x_0|_w$. Since $|E-x_0|_w>|E|_w$ we have that $s>r$ and using that $E-x_0$ passes through the origin we find
\[
\int_{\pa E}H_{\pa E}f(|x|)\,d\H^1
\leq \int_{\pa (E-x_0)}H_{\pa E}f(|x|)\,d\H^1 
\leq 2\pi f(s) \leq 2\pi f(r).
\]

\end{proof}

We note that Proposition \ref{prop:curv} remains true if we replace the convexity assumption on $E$ with the assumption that $E$ is starshaped with respect to the origin.
Next result shows that when $E$ is a convex set containing the origin, the inequality above can be given in a stronger quantitative form. To this aim, given any sufficiently smooth set $E\subset \R^2$,
we introduce the following positive quantities
\beq
\alpha_f(E)= -\int_{\pa E} 
 \left(|x|-\frac{\langle x,\nu \rangle^2}{|x|} \right)    f'(|x|)    \,d\H^1,
\eeq
and 
\beq
\beta_f(E)= \int_{\pa E} \left(|x|-\frac{\langle x,\nu \rangle^2}{|x|} \right)     \left(\frac{f(|x|)- |x|f'(|x|)}{|x|^2} \right)   \,d\H^1.
\eeq

\begin{thm}\label{teo:2d}
Let $E$ be a convex set of class $C^{1,1}$ containing the origin such that $|E|_w=|B_r|_w$. Then
\beq \label{eq:curv2d1}
\alpha_f(E) \leq \int_{\pa B_r} H_{\pa B_r} f(|x|)\, d\H^1 -\int_{\pa E} H_{\pa E}f(|x|)\, d\H
^1 \leq \beta_f (E).
\eeq
\end{thm}
\begin{proof}
Denote by $\rho: \R\to (0,\infty)$ a $C^{1,1}$ periodic function such that $\pa E= \{\rho(\theta)( \sin \theta, \cos \theta):\, \theta \in [0,2\pi]\}$.
To prove the first inequality we observe that
\beq \label{eq:normal}
|x|- \frac{\langle x,\nu \rangle^2}{|x|} = \frac{\rho \rho '^2}{\rho^2+\rho'^2}. 
\eeq
Using that 
\[
t\arctan t \geq \frac{ t^2}{\sqrt{1+t^2}} 
\]
for all $t\in \R$, using \eqref{eq:curvature1} and arguing as in the proof of Proposition \ref{prop:curv} we get
\[
\begin{split}
\int_{\pa E} H_{\pa E} f(|x|)\, d\H^1
=&\int_{0}^{2\pi} \rho^2 \frac{\rho'}{\rho} \arctan\left(\frac{\rho'}{\rho}\right)\ f'(\rho) d\theta +2\pi f(r)\\
\leq&\int_{0}^{2\pi} \frac{\rho' \rho^2}{\sqrt{\rho'^2+\rho^2}}\ f'(\rho) d\theta +2\pi f(r)
\\
=& \int_{\pa E}  \left(|x|-\frac{\langle x,\nu\rangle^2 }{|x|}\right) f'(|x|)\, d\H^1  + \int_{\pa B_r} H_{\pa B_r} f(|x|) \, d\H^1.
\end{split}
\]
To prove the second inequality we first recall that up to a constant
$u(x)=\log |x|$ is the fundamental solution of the laplacian in two dimensions. As a consequence of this fact we claim that if $E$ contains the origin and $|E|_w=|B_r|_w$ 
\beq \label{eq:weight}
\int_{\pa E} \frac{1}{|x|}f(|x|) \, d\H^1\geq \int_{\pa B_r}\frac{1}{|x|}f(|x|) \, d\H^1.
\eeq
In fact, using the divergence theorem in $E\setminus B_\e$, where $\overline B_\e\subset E$ and letting $\e\to 0^+$, with some elementary calculations we get 
\beq \label{eq:above}
\begin{split}
    \int_{\pa E} \frac{1}{|x|}f(|x|) \, d\H^1\geq \int_{\pa E} \frac{\langle x,\nu\rangle }{|x|^2}f(|x|)\, d\H^1
&= \int_E \Div\left( \frac{x}{|x|^2}  f(|x|) \right)dx
\\
&= 2\pi f(0)+\int_{E} \frac{f'(|x|)}{|x|} \, dx
\\&  
=\int_{\pa B_r}\frac{1}{|x|}f(|x|) \, d\H^1,
\end{split}
\eeq
where in the last equality we used the assumption that $|E|_w=|B_r|_w$.
Note that the above inequality is strict, unless $E=B_r$ and it can be actually written in a quantitative form arguing, see for instance \cite{LAM1}. 
To prove the proposition we now use \eqref{eq:above} and the diverge theorem on a manifold.
\[
\begin{split}
    \int_{\pa E} H_{\pa E} f(|x|)& \,d\H^1\geq \int_{\pa E} H_{\pa E}\left\langle\frac{x}{|x|},\nu\right\rangle  f(|x|)\, d\H^1\\
&= \int_{\pa E} \Div_\tau \left( \frac{x}{|x|}f(|x|)\right)  \,d\H^1
\\
&=\int_{\pa E}\frac{1}{|x|}f(|x|)\,d\H^1
+\int_{\pa E}\left(|x|-\frac{\langle x,\nu \rangle^2}{|x|} \right) \left(\frac{|x|f'(|x|)- f(|x|)}{|x|^2} \right) \,d\H^1
\\
&\geq 
\int_{\pa B_r}\frac{1}{|x|}f(|x|) \, d\H^1 +\int_{\pa E}\left(|x|-\frac{\langle x,\nu \rangle^2}{|x|} \right)  \left(\frac{|x|f'(|x|)- f(|x|)}{|x|^2} \right) \,d\H^1
\\
&=
\int_{\pa B_r}H_{\pa B_r}f(|x|) \, d\H^1 +\int_{\pa E}\left(|x|-\frac{\langle x,\nu \rangle^2}{|x|} \right)  \left(\frac{|x|f'(|x|)- f(|x|)}{|x|^2} \right) \,d\H^1,
\end{split}
\]
thus proving the second inequality in \eqref{eq:curv2d1}.
\end{proof}
\begin{oss}
Observe that the above proof shows that 
inequality \eqref{eq:weight} holds for any set $E$ of finite perimeter containing the origin in the interior. Note however that this latter assumption can not be weakened as it is shown by an example in \cite{FuLa}.
\end{oss}
Note that the above theorem essentially says that one may control the gap $\He(E)-\He(B_r)$ with the oscillation of the normals to $E$ and $B_r$. 
To be more precise,  
let us denote by $\pi$ the projection of $\pa E$ on $\pa B_r$. Observe that 
\[
\frac{1}{2}|x| |\nu_{\pa E} (x)- \nu_{\pa B_r} (\pi(x))|^2\leq 
|x|-\frac{\langle x,\nu \rangle^2}{|x|}\leq |x| |\nu_{\pa E} (x)- \nu_{\pa B_r} (\pi(x))|^2.
\]
Then it is clear that under the assumption of Theorem \ref{teo:2d} and if
 $E\subset B_R$ for some $R>0$, then  
 \[|\int_{\pa E} H_{\pa E} f(|x|)\, d\H^1- \int_{\pa B_r} H_{\pa B_r} f(|x|)\, d\H^1| \leq C(f,R)\|\nu_{\pa E}(x)- \nu_{\pa B_r}(\pi (x))\|^2_{L^2(\pa E)}
 \] 
 where the constant $C(f,f',R)$ depends only on the function $f$ and its derivative and $R$.
\\
Note that the above Theorem applies in particular to the Gaussian weight $\gamma(r)= e^{-\frac{r^2}{2}}$. As a consequence of this we get

\begin{cor}
Let $E\subset \R^2$ convex and passing through the origin and $r>0$ such that $\gamma(E)=\gamma(B_r)$. Then it holds
\[
 \alpha_\gamma (E)\leq
\int_{\pa B_r}H_{\pa B_r}e^{-\frac{|x|^2}{2}}\, d\H^1-\int_{\pa E}H_{\pa E}e^{-\frac{|x|^2}{2}}\, d\H^1 \leq \beta_\gamma(E).
\]
\end{cor}

\begin{oss}We note that inequalities \eqref{eq:curvature2d} and \eqref{eq:curv2d1} can be immediately extended to any bounded convex set $E$ contained in the plane with not empty interior. To see this we recall (see \cite[Section 4.2]{Sc}) that for such $E$ there exists a \textit{curvature measure} $\mu_E$ supported on $\pa E$ 
such that if $E_h$ is a sequence of smooth convex sets converging in the Hausdorff distance to $E$, then 
$H_{\pa E_h} \, \H^1 \res \pa E_h  \rightharpoonup \mu_E$.
 Using this measure, for instance \eqref{eq:curvature2d} becomes 
 \[
 \int_{\pa E} f(|x|)\, d\mu_E\leq \int_{\pa B_r} H_{\pa B_r} f(|x|)\, d\H^1,
 \]
 whenever $E$ is such that $|E|_w=|B_r|_w$.
A similar extension also holds for
  \eqref{eq:curv2d1}.
 
\end{oss}
 We conclude this section by proving another consequence of Theorem \ref{teo:2d}.  More precisely, if $E_h\to B_r$ in the Hausdorff distance then the corresponding weighted curvature integrals converge with a speed controlled by the distance of $E_h$ from $B_r$.
 To this aim, given two closed sets $E,F\subset \R^2$ we denote by $d_\H (E,F)$ the Hausdorff distance between $E$ and $F$.
We will also use the following lemma, which is the two dimensional version of a more general statement proved in \cite{Fus} (see proof of Lemma 3.3). 
 \begin{lem}\label{lem:convex}
 Let $E\subset \R^2$ be a convex body containing the origin and let $\rho : [0,2\pi] \to (0,2)$ be such that $\pa E= \rho(\theta)(\cos \theta,\sin \theta)$. Then
 \[
 \|\rho'\|_{L^\infty} \leq 2\sqrt{\|\rho-1\|_{L^\infty}} \frac{1+\|\rho-1\|_{L^\infty}}{1-\|\rho-1\|_{L^\infty}} .
 \]
 \end{lem}
 
\begin{thm}
Let $f:[0,\infty) \to (0,\infty)$ a $C^1$ not increasing function and $E_h \subset \R^2$ a sequence of convex set converging to $B_r$ in the Hausdorff distance. Then there exists a constant depending only on $r$ and $f$ such that for $h$ large
\[
\left| \int_{\pa E_h} f(|x|) d\mu_{E_h}- \int_{\pa B_r}f(|x|)H_{\pa B_r} d\H^1 \right| \leq 
C d_{\H}(E_h, B_r). 
\]
\end{thm}
\begin{proof}
Let $w$ be the function defined as in \eqref{eq:defw} and for all $h$ let $r_h$ be the unique positive number such that $|B_{r_h}|_w=|E_h|_w$.
We have

\beq\label{eq:curvature2}
\begin{split}
\left|\int_{\pa B_r}f(|x|)H_{\pa B_r}\,d\H^1 -\int_{\pa E_h} f(|x|)\,d\mu_{E_h} \right|
\leq& 
\left|\int_{\pa B_r}f(|x|) H_{\pa B_r}\, d\H^1-\int_{\pa B_{r_h}}f(|x|)H_{\pa B_r}\, d\H^1  \right|\\
&+  \int_{\pa B_{r_h}}H_{\pa B_{r_h}}f(|x|) \, d\H^1-\int_{\pa E_h}f(|x|)\, d\mu_{E_h}. 
\end{split}
\eeq
Setting $d_h= d_\H(E_h,B_r)$, since for $h$ large $B_{r-d_h}\subset E_h\subset B_{r+d_h}$ we have 
$r-d_h\leq r_h \leq r+d_h$, that is $|r-r_h|\leq d_h$. Hence for $h$ large we may estimate the first integral on the right hand side of \eqref{eq:curvature2} as follows.
\[
\begin{split}
\left|\int_{\pa B_r}f(|x|) H_{\pa B_r}\, d\H^1-\int_{\pa B_{r_h}}f(|x|)H_{\pa {B_r}}\, d\H^1 \right|
&\leq 2\pi |f(r_h)-f(r)|
\\& \leq 2\pi \max_{[r/2,2r]} |f'| |r_h-r| \leq C d_h.
\end{split}
\] 
To estimate the second integral we denote by $\rho_h$ the Lipschitz function such that $\pa E_h = \rho_h (\cos \theta, \sin \theta)$. 
Then we use the second inequality in \eqref{eq:curv2d1}, \eqref{eq:normal} and Lemma \ref{lem:convex} applied to $\frac{1}{r}E_h$ to get for $h$ large
\[
\begin{split}
 \int_{\pa B_{r_h}}H_{\pa B_{r_h}}f(|x|) \, d\H^1-&\int_{\pa E_h}f(|x|)\, d\mu_{E_h}  
 \\
 &\leq \max_{\rho \in [r/2,2r]} \left(\frac{f(\rho)-\rho f'(\rho)}{\rho'^2}\right)
 \int_{\pa E_h}\left( |x|- \frac{\langle x,\nu_{\pa E_h}\rangle^2}{|x|}\right)\, d\H^1 
 \\&
\leq C\int_{0}^{2\pi} \rho_h'^2 \, d\theta
\leq 8\pi Cr \|\rho_h-r\|_{\infty}\left(\frac{r+ \|\rho_h-r\|_\infty}{r-\|\rho_h-r\|_\infty}\right)^2 
\\
&\leq C' d_h. 
\end{split}
\]
This last estimates concludes the proof.
\end{proof}

Observe that arguing as in final part of the above proof under the assumption of  
  Theorem \ref{teo:2d} if $d_\H(E,B_r)<1$, we have 
  \[\left|\int_{\pa E} H_{\pa E} f(|x|)\, d\H^1- \int_{\pa B_r} H_{\pa B_r} f(|x|)\, d\H^1\right| \leq C(f)d_\H(E,B_r)
 \] 
for some constant depending only on $f$.
\section{Higher dimension}\label{sec:highd}
The isoperimetric inequality proved in Proposition \ref{prop:curv} is false in higher dimension, as shown by the following example. 
\begin{exa}
Let $n=3$, $r>0$. If $r$ is sufficiently small, there exists a smooth convex body $E$ such that $\gamma (E)= \gamma (B_r)$ but
\[
\He (E) > \He (B_r).
\]
\end{exa}
\begin{proof}
Denote by $C(t)$ the cylinder in 
\[
C(s)=\{(x',x_3)\in \R^2\times \R:\, |x'|\leq s\}
\]
For any $r>0$ let $s(r)$ the unique positive number such that 
$
\gamma (C_{s(r)})=
\gamma(B_r).
$
Note that 
$$
\gamma (C_s)=\int_0^s te^{-\frac{t^2}{2}}\, dt= 1-e^{-\frac{s^2}{2}}
$$
while 
\[
(2\pi)^{\frac12}\gamma(B_r)= 2\int_0^r t^2e^{-\frac{t^2}{2}}\, dt= 
2\left(-re^{-\frac{r^2}{2}}  +\int_0^r    e^{-\frac{t^2}{2}}\, dt      \right).
\]
Since 
$\gamma(B_r)= \gamma (C_s)$ we get
$$
e^{\frac{-s^2}{2}} =1- \frac{2}{(2\pi)^\frac12}\left(-re^{-\frac{r^2}{2}}  +\int_0^r    e^{-\frac{t^2}{2}}\, dt      \right)
$$
Moreover, we also have
\[
\He(C_s)= (2\pi)^{\frac32} e^{-\frac{s^2}{2}},
\qquad\He (B_r)= 8\pi re^{-\frac{r^2}{2}}.
\]
Hence
\[
\begin{split}
    \He(C_s) &= 
(2\pi)^{\frac32} +4\pi (r e^{-\frac{r^2}{2}} -\int_{0}^r e^{-\frac{t^2}{2}}\, dt)
\\
&=\He (B_r)+(2\pi)^{\frac32}-4\pi (r e^{-\frac{r^2}{2}} +\int_{0}^r e^{-\frac{t^2}{2}}\, dt)
\\
&\geq \He (B_r)+ 1,
\end{split}
\]
provided $r$ is sufficiently small.
Let $C_{T,s(r)}$ the convex body 
obtained as the union of the cylinder $C_{s(r)} \cap \{|x_3|<T\}$ with the two half balls of radius $s(r)$ placed on the upper and lower basis of the cylinder. 
Since \[\gamma(C_{T,s(r)})\to \gamma (C_{s(r)}) \qquad \ \]
as $T\to \infty$, we conclude that $\He(C_{T,s(r)})>\He (B_{r'})$ with $r'$ such that $\gamma(C_{T,s(r)})=\gamma(B_{r'})$, provided $r$ is small and $T$ is sufficiently large.
\end{proof}

\begin{lem}
Let $E$ be a bounded open set of class $C^2$ starshaped with respect to the origin and let $h:\S^{n-1}\to (0,\infty)$ a $C^2$ function such that 
\[
\pa E= \{y=x h(x),\, x\in \S^{n-1} \}.
\] 
Then
\beq \label{eq:curv3}
H_{\pa E} (xh(x))=\frac{-\frac{1}{h}\Delta_{\S^{n-1}} h  +n-1}{\sqrt{|\nabla h|^2 +h^2}} 
+\frac{h\frac{1}{2}\langle\nabla|\nabla h|^2,\nabla h\rangle +h^2|\nabla h|^2}{h^2\sqrt{(|\nabla h|^2 +h^2})^3}.
\eeq
\end{lem}
\begin{proof}
First, we extend $h$ to $\R^n\setminus \{0\}$ as a homogeneous function of degree $0$ still denoted by $h$. Note that with this definition for any $x\in \S^{n-1}$ the tangential gradient of $h$ at $x$ coincides with the gradient of $h$ at the same point.
Note also that the exterior normal to $\pa E$ at $xh(x)$, for $x\in \S^{n-1}$, is given by
\[
\nu  (xh(x))= \frac{xh(x)- \nabla h(x)}{\sqrt{h^2(x)+\nabla h(x)|^2} },
\]
 where we have set $\nu = \nu_{\pa E}$. 
Thus, setting $y=xh(x)$ and recalling \eqref{eq:defnH} we have
\[
H_{\pa E} (y)= \Div \nu(y)=  \frac{\pa \nu_i}{\pa y_i}(y)=  \frac{\pa \nu_i}{\pa x_j} (y) \frac{\pa x_j}{\pa y_i}(y),
\]
where we have adopted the standard convention of summation over repeated indexes.
Since the derivatives of $h$ are homogeneous of degree $-1$ we have
\[
\frac{\pa x_j}{\pa y_i}= \frac{\pa }{\pa y_i} \frac{y_j}{h(y)}=
\frac{\delta_{ij}}{h(x)} - \frac{x_j}{h^2(x)} \frac{\pa h}{\pa x_i}(x). 
\]
Hence 
\[
H_{\pa E} (xh(x))= \frac{1}{h(x)} \Div \left(\frac{x h(x) -\nabla h (x)}{\sqrt{h^2+|\nabla h|^2}}\right)- \frac{\pa \nu_i}{\pa x_j} \frac{x_j}{h^2} \frac{\pa h}{\pa x_i}.
\]
Denoting by $\Div_\tau$ the tangential divergence on $\S^{n-1}$ we have
\[\begin{split}
H_{\pa E} (xh(x))&=
 \frac{1}{h(x)} \Div \left(\frac{x h(x) -\nabla h (x)}{\sqrt{h^2+|\nabla h|^2}}\right)= \frac{1}{h(x)} \Div_\tau \nu(x) + \frac{1}{h} \frac{\pa \nu_i}{\pa x_j}x_ix_j - \frac{\pa \nu_i}{\pa x_j} \frac{x_j}{h^2} \frac{\pa h}{\pa x_i}\\&
 =
 \frac{1}{h(x)} \Div_\tau \nu(x) +  \frac{1}{h^2} (h^2 +|\nabla h|^2)^\frac12  \frac{\pa \nu_i}{\pa x_j} \nu_i x
 =  \frac{1}{h(x)} \Div_\tau \nu(x) .
 \end{split}
\]
Then, since $\langle x,\nabla h\rangle =0$, an easy calculation gives \eqref{eq:curv3}.
\end{proof}
We conclude by proving that in higher dimension if $r>\sqrt{n-2}$ then the ball $B_r$ is a local maximizer of the integral of the weighted mean curvature with respect to $C^2$ perturbations. Quite surprisingly, the ball $B_r$ is a local minimizer if $r$ is small enough.
\begin{thm} \label{teo:localmax}
For all $r>0$ there exist $\e_0(r),C(r)>0$ with the property that if $u\in W^{2,\infty}(\S^{n-1})$, $\|u\|_{W^{2,\infty}}\leq \e<\e_0$ 
 and
 $E=\{trx(1+u(x)),\,\, x\in \S^{n-1},\, t\in (0,1)\}$ is such that $\gamma(E)=\gamma(B_r)$ 
 then 
 \beq \label{eq:fug3}
\He(B_r)-\He(E) \geq r^{n-2}e^{-\frac{r^2}{2}} \left(r^2-n+2-C\e_0\right)\|u\|_{W^{1,2}(\S^{n-1})}.
\eeq
Moreover, if $E=-E$ 
\beq \label{eq:fug4}
\He(B_r)-\He(E)
\leq r^{n-2}e^{-\frac{r^2}{2}}
\left(C\e +r^2-(n-2)\frac{n-1}{2n}\right)\|u\|_{W^{1,2}(\S^{n-1})}.
\eeq
\end{thm}
\begin{proof}
To prove our statement we use \eqref{eq:curv3} with $h$ replaced by $r(1+u)$, thus getting
\[\begin{split}
&\He(E)= \int_{\pa E} H_{\pa E}\, d\H^{n-1}
\\
&=\int_{\S^{n-1}} 
\left[(n-1)-\left(\frac{\Delta u }{1+u}\right)
+
\left(\frac{\langle\nabla^2 u \nabla u,\nabla u\rangle +(1+u)|\nabla u|^2}{(1+u)(|\nabla u|^2 +(1+u)^2)}\right)\right]\frac{r^{n-2}e^{-\frac{r^2}{2}(1+u)^2}}{((1+u))^{2-n}}\, d\H^{n-1}.
\\
&=r^{n-2} \left[(n-1)I -J+K        \right].
\end{split}
\]

By Taylor expansion and the smallness of $u$ we get
\[
\begin{split}
I=&
 \int_{\S^{n-1}}(1+u)^{n-2}  e^{-\frac{r^2}{2}(1+u)^2}d\H^{n-1}
\\
=&e^{-\frac{r^2}{2}}\left(n\omega_n 
+((n-2)-r^2)\int_{\S^{n-1}} u  \,d\H^{n-1}\right)
\\
&+e^{-\frac{r^2}{2}}\left(\frac{(n-2)(n-3)-(2n-3)r^2
+r^4)}{2} \int_{\S^{n-1}}  u^2 d\H^{n-1}\right) +o(\|u\|^{2}_{L^2(\S^{n-1})})
\end{split}
\]
and
\[
\begin{split}
    e^{\frac{r^2}{2}}\int_{\S^{n-1}}(1+u)^{n-3}& \Delta u e^{-\frac{r^2(1+u)^2}{2}}\, d\H^{n-1}\\
    =&\int_{\S^{n-1}}\Delta u d\H^{n-1}
    +(n-3-r^2)\int_{\S^{n-1}}u\Delta u\, d\H^{n-1}
    \\
    &+ \int_{\S^{n-1}} u^2\Delta u G(u)\,d\H^{n-1},
\end{split}
\]
where $G(u)$ contains the remainder in the Taylor expansion. 
Using that $\int \Delta u \,d\H^{n-1}=0$ and integrating by parts the terms involving the Laplace-Beltrami operator we infer
\[
J=\int_{\S^{n-1}}(1+u)^{n-3} \Delta u e^{-\frac{r^2(1+u)^2}{2}}=
-e^{-\frac{r^2}{2}}(n-3-r^2)\int_{\S^{n-1}}|\nabla u|^2 d\H^{n-1} +o(\|u\|^2_{W^{1,2}(\S^{n-1})}).
\]
The last term is actually easier to treat since we are assuming also the smallness of the Hessian of $u$. Thus we have 
\[
\begin{split}K=&
\int_{\S^{n-1}} \frac{ \langle \nabla^2 u, \nabla u,\nabla u\rangle+(1+u) |\nabla u|^2 }{(1+u)^{1-n}((1+u)^2+|\nabla u|^2)}
 e^{-\frac{r^2}{2  }(1+u)^2}\,
d\H^{n-1}
\\=&
e^{-\frac{r^2}{2}}(1+o(\|u\|^2_{W^{1,2}(\S^{n-1})}))\int_{\S^{n-1}}\langle \nabla^2 u, \nabla u,\nabla u \rangle+|\nabla u|^2\, d\H^{n-1}.
\end{split}
\]
Collecting all the previous equalities we then get
\beq\begin{split}\label{eq:fug0}
\He(E)&-\He(B_r)
=
r^{n-2}e^{-\frac{r^2}{2}}(n-1)\left(((n-2)-r^2)\int_{\S^{n-1}} u  \,d\H^{n-1}\right)\\
&+r^{n-2}e^{-\frac{r^2}{2}}(n-1)\left(\frac{(n-2)(n-3)-(2n-3)r^2
+r^4)}{2} \int_{\S^{n-1}}  u^2 d\H^{n-1}\right)
\\
&+r^{n-2}e^{-\frac{r^2}{2}}\left((n-2-r^2)\int_{\S^{n-1}}|\nabla u|^2 d\H^{n-1}+\int_{\S^{n-1}}\langle \nabla^2 u, \nabla u,\nabla u \rangle\, d\H^{n-1} \right).
\end{split}
\eeq
To estimate the integral of $u$ in the previous equation we need to exploit the assumption that the Gaussian measures of $E$ and $B_r$ are equal. 
In fact, since
\beq
\gamma(B_r)=\gamma(E) 
=\frac{r^{n}}{(2\pi)^{n/2}}\int_B(1+u(x))^ne^{-\frac{r^2|x|^2(1+u(x))^2}{2}}\,dx,
\eeq
we can expand the integral via Taylor formula to find
$$
\int_0^1t^{n-1}\,dt\int _{\mathbb{S}^{n-1}}\Big[(1+u)^ne^{-\frac{r^2t^2(1+u)^2}{2}}-e^{-\frac{r^2t^2}{2}}\Big]\,d\H^{n-1}=0.
$$
Using again Taylor expansion, we then easily get
\beq \label{eq:fug}
\begin{split}
&0 
= \int_0^1t^{n-1}e^{-\frac{r^2t^2}{2}}\,dt\int _{\mathbb{S}^{n-1}}\big[(1+u)^ne^{-r^2t^2(u+u^2/2)}-1\big]\,d\H^{n-1} \\
&= \int_0^1t^{n-1}e^{-\frac{r^2t^2}{2}}\,dt\int _{\mathbb{S}^{n-1}}\Big[(n-r^2t^2)u+\Big(\frac{n(n-1)}{2}-\frac{(2n+1)r^2t^2}{2}+\frac{r^4t^4}{2}\Big)u^2\Big]\,d\H^{n-1}\\
&\qquad\qquad\qquad\qquad\qquad\qquad\qquad\qquad\qquad\qquad\qquad\qquad +o(\|u\|^2_{L^2(\S^{n-1})})  \\
&= \int _{\mathbb{S}^{n-1}}\Big[(na_n-r^2b_n)u+\Big(\frac{n(n-1)a_n}{2}-\frac{(2n+1)r^2b_n}{2}+\frac{r^4c_n}{2}\Big)u^2\Big]\,d\H^{n-1}
\\&
\qquad\qquad\qquad\qquad\qquad\qquad\qquad\qquad\qquad\qquad\qquad\qquad +o(\|u\|^2_{L^2(\S^{n-1})}) , 
\end{split}
\eeq
where we have set
$$
a_n=\int_0^1t^{n-1}e^{-\frac{r^2t^2}{2}}\,dt,\qquad b_n=\int_0^1t^{n+1}e^{-\frac{r^2t^2}{2}}\,dt,\qquad c_n=\int_0^1t^{n+3}e^{-\frac{r^2t^2}{2}}\,dt.
$$
A simple integration by parts gives that
$$
b_n=\frac{na_n}{r^2}-\frac{e^{-\frac{r^2}{2}}}{r^2},\qquad c_n=\frac{n(n+2)a_n}{r^4}-\frac{(n+2)e^{-\frac{r^2}{2}}}{r^4}-\frac{e^{-\frac{r^2}{2}}}{r^2}.
$$
Thus, inserting the above values of $b_n$ and $c_n$ into \eqref{eq:fug} we arrive at
\begin{equation}\label{eq:fug1}
\int _{\mathbb{S}^{n-1}}u\,d\H^{n-1}=-\frac{n-1-r^2}{2}\int _{\mathbb{S}^{n-1}}u^2\,d\H^{n-1}+o(\|u\|_{L^2(\S^{n-1})}^2).
\end{equation}
Hence, \eqref{eq:fug0} combined with \eqref{eq:fug1} gives
\beq\label{eq:fug2}
\begin{split}
\He(E)-\He(B_r)
=&-r^{n-2}e^{-\frac{r^2}{2}}(n-1)(n-2)\int _{\mathbb{S}^{n-1}}u^2\,d\H^{n-1}
\\
& +(n-2-r^2) r^{n-2}e^{-\frac{r^2}{2}} \int_{\S^{n-1}}|\nabla u|^2 d\H^{n-1}  
\\
&
+r^{n-2}e^{-\frac{r^2}{2}}\int_{\S^{n-1}} \langle \nabla^2 u\nabla u ,\nabla u\rangle \, d\H^{n-1}+
o(\|u\|^2_{W^{1,2}(\S^{n-1})}).
\end{split}
\eeq
From this equality we immediately conclude the proof of \eqref{eq:fug3}.

To prove the second inequality 
for any integer $k \geq 0$ we denote by $y_{k,i}$, $i= 1, \dots, G(n,k)$, the spherical harmonics of order $k$, i.e., the restrictions to $\S^{n-1}$ of the homogeneous harmonic polynomials of degree $k$, normalized so that $|| y_{k,i} ||_{L^2(\S^{n-1})}=1$.
 The functions $y_{k,i}$ are eigenfunctions of the Laplace-Beltrami operator on $\S^{n-1}$ and for all $k$ and $i$
\[
- \Delta_{\S^{n-1}}y_{k,i} = k(k+n-2)y_{k,i}\,.
\]
Therefore if we write 
\[
u = \sum_{k=0}^{\infty} \sum_{i=1}^{G(n,k)}a_{k,i}y_{k,i}, \quad \mbox{where} \quad a_{k,i}= \int_{\S^{n-1}} u y_{k,i} d\H^{n-1},
\]
we have
\begin{equation} \label{eq:armsph}
|| u ||^2_{L^2(\S^{n-1})} = \sum_{k=0}^{\infty} \sum_{i=1}^{G(n,k)}a^2_{k,i}, \quad || D_{\tau}u ||^2_{L^2(\S^{n-1})} = \sum_{k=1}^{\infty}k(k+n-2) \sum_{i=1}^{G(n,k)}a^2_{k,i}\,.
\end{equation}
Note that \eqref{eq:fug} implies
\[
|a_{0}|^2=\left|\int_{\S^{n-1}} u\,d\H^{n-1}\right|^2 = o(\|u\|_{L^2(\S^{n-1})}).
\]
Since $E=-E$ we also have that $u$ is an even function, hence $a_{2k+1,i}=0$ for all $k \in \mathbb{N}$ and $i \in \{1,\dots, G(2k+1,n)\}$. Hence we can write
\[
\|\nabla u\|^2\geq 2n\|u\|_{L^2(\S^{n-1})} -o(\|u\|^2_{L^2(\S^{n-1})})
\]
which finally gives
\[\begin{split}
    \int_{\pa E} H_{\pa E}e^{-\frac{|x|^2}{2}}\, d\H^{n-1} -&\int_{\pa B_r} H_{\pa B_r}e^{-\frac{|x|^2}{2}}\, d\H^{n-1}\\
    &\geq r^{n-2}e^{-r^2}\left((n-2)\frac{n+1}{2n}        -r^2-\e_0\right)\|\nabla u\|_{L^2(\S^{n-1})}.
\end{split}
\]
\end{proof}
The next result is a local maximality results under weaker assumptions. To this aim we introduce the function
\[
\psi (s)= \frac{1}{\sqrt{2\pi}} \int_{-\infty}^s e^{-\frac{t^2}{2}}\, dt,
\]
which is the value of the Gaussian volume of the half space $H_s=\{x\in \R^n : x_1 \leq s\}$.

\begin{thm}  \label{teo:localmax1}
Let $n\geq 3$, $M>0$ and $m \geq \max \{\psi (2M), \psi (\sqrt {n-2}) \}$.
For any $C^2$ convex set containing the origin with $\gamma(E)=\gamma(B_r)=m$ and $\|H_{\pa E}\|_{L^\infty} \leq M$ it holds
\beq\label{eq:localmax1}
\He(E) \leq \He (B_r).
\eeq
Moreover, if $m\geq \psi(\sqrt{n-2})$ 
\beq \label{eq:localmax3}
\int_{\pa E}  \frac{\langle x,\nu\rangle}{|x|}   H_{\pa E}  e^{-\frac{|x|^2}{2}}\leq \int_{\pa B_r}  \frac{\langle x,\nu\rangle}{|x|}   H_{\pa B_r}  e^{-\frac{|x|^2}{2}}
\eeq
for any convex set $E$ containing the origin with $\He(E)<\infty$ and $\gamma(E)=\gamma(B_r)=m.$
\end{thm}
\begin{proof}
Let
$E$ as in the statement and let $r_E$ the radius of the largest ball centered at the origin and contained in $E$, i.e.
\beq \label{eq:defrE}
r_E= \sup\{r: B_r \subset E\}. 
\eeq
Let $x \in \pa B_{r_E}\cap \pa E$ and let $H$ be the halfspace containing the origin and such that the hyperplane $\pa H$ is tangent to $E$ at $x$. Since by convexity $E\subset H$ we have $\psi (r_E)=\gamma(H)\geq m$, hence $r_E\geq \psi^{-1}(m)$. Therefore, our
 assumption on $m$ implies $ r_E \geq \max\{ 2M, \sqrt{2(n-2)}\}$.
Now, using the divergence theorem on mainfolds we infer
\[\begin{split}
\He (E)&= \int_{\pa E}H_{\pa E} \frac{\langle x,\nu\rangle}{|x|} e^{-\frac{|x|^2}{2}}\, d\H^{n-1} +\int_{\pa E} H_{\pa E} \left(1-\frac{\langle x,\nu\rangle}{|x|}\right) e^{-\frac{|x|^2}{2}}\, d\H^{n-1}
\\
&=
\int_{\pa E}\Div_\tau\left( \frac{x}{|x|} e^{-\frac{|x|^2}{2}}\right)\, d\H^{n-1} +\int_{\pa E} H_{\pa E} \left(1-\frac{\langle x,\nu\rangle}{|x|}\right) e^{-\frac{|x|^2}{2}}\, d\H^{n-1}.
\end{split}
\]
We compute the tangential divergence to find
\[
\Div_\tau\left( \frac{x}{|x|} e^{-\frac{|x|^2}{2}}\right)= \frac{n-1}{|x|}e^{-\frac{|x|^2}{2}} - \left(1-\frac{\langle x,\nu\rangle^2}{|x|^2} \right)\left(\frac{1}{|x|} + |x|\right) e^{-\frac{|x|^2}{2}}.
\]
This gives 
\beq\label{eq:maximality}
\begin{split}
\He (E)=&\int_{\pa E} \frac{n-1}{|x|}e^{-\frac{|x|^2}{2}} \, d\H^{n-1}
\\&+\int_{\pa E}\left(1-\frac{\langle x,\nu\rangle}{|x|} \right)\left( H_{\pa E} -\left( \frac{1}{|x|} + |x|\right) \left(1+ \frac{\langle x,\nu\rangle }{|x|}\right)\right) e^{-\frac{|x|^2}{2}}\, d\H^{n-1}
\\
=&\int_{\pa E} \frac{n-1  }{|x|^2} \langle x,\nu\rangle e^{-\frac{|x|^2}{2}} \, d\H^{n-1}
\\
&+\int_{\pa E}\left(1-\frac{\langle x,\nu\rangle}{|x|} \right)\left( H_{\pa E} +\frac{n-1}{|x|}-\left( \frac{1}{|x|} + |x|\right) \left(1+ \frac{\langle x,\nu\rangle }{|x|}\right)\right) e^{-\frac{|x|^2}{2}}\, d\H^{n-1}
 \end{split}
 \eeq
Note that, differently from the two dimensional case, the integral quantity
$\int_{\pa E} \frac{n-1  }{|x|^2} \langle x,\nu\rangle e^{-\frac{|x|^2}{2}} \, d\H^{n-1}$
 is maximized by the ball centered at the origin with the same Gaussian volume of $E$.
  Indeed, using the divergence theorem
\beq\label{eq:maximality1}
\begin{split}
\int_{\pa E} \frac{1}{|x|^2} \langle x,\nu\rangle e^{-\frac{|x|^2}{2}} \, d\H^{n-1}
=& \int_{E} \Div \left( \frac{x}{|x|^2} e^{-\frac{|x|^2}{2}} \, dx\right) \\
=& \int_{E} \frac{n-2}{|x|^2} e^{-\frac{|x|^2}{2}} \,dx -(2\pi)^{\frac{n}{2}}\gamma(E)\\
\leq& \int_{B_r} \frac{n-2}{|x|^2} e^{-\frac{|x|^2}{2}} \,dx -(2\pi)^{\frac{n}{2}}\gamma(B_r)\\
=& \int_{\pa B_r} \frac{1}{|x|}  e^{-\frac{|x|^2}{2}} \, d\H^{n-1}= \frac{1}{n-1} \He(B_r).
\end{split}
\eeq
Since $E$ is a convex set containing the origin, we have that $\langle x,\nu \rangle \geq 0$
for all $x\in \pa E$.
This fact together with \eqref{eq:maximality} and \eqref{eq:maximality1} leads to
\beq \label{eq:localmax2}
\He (E) \leq \He(B_r) +\int_{\pa E}\left(1-\frac{\langle x,\nu\rangle}{|x|} \right)\left( H_{\pa E} +\frac{n-2}{|x|} -|x|\right) e^{-\frac{|x|^2}{2}}\, d\H^{n-1}.
\eeq
Since $H_{\pa E} (x)\leq M \leq r_E/2$, $r_E\leq |x|$ for all $x\in \pa E$
the assumption $r_E\geq 2\sqrt{n-2}$ implies that the integrand on the right hands side is negative, which in turn gives \eqref{eq:localmax1}. 
Inequality \eqref{eq:localmax3} is also a consequence of \eqref{eq:localmax2}. Indeed, \eqref{eq:localmax2} implies that if 
$r\geq \psi(\sqrt{n-2})$
\[
\int_{\pa E}  \frac{\langle x,\nu\rangle}{|x|}   H_{\pa E}  e^{-\frac{|x|^2}{2}}\, d\H^{n-1}\leq  \He(B_r)=\int_{\pa B_r}  \frac{\langle x,\nu\rangle}{|x|}   H_{\pa B_r}e^{-\frac{|x|^2}{2}}\, d\H^{n-1}.
\]

\end{proof}

\bibliographystyle{plain}
\bibliography{references}

\end{document}